\pdfoutput=1
\documentclass[twocolumn,noversion,squeezed]{cdcarticle}

\usepackage{tikz}
\usepackage{graphicx}
\graphicspath{{figures/}}

\usepackage{amsmath,amssymb,amsfonts}   
\usepackage[utf8]{inputenc}				      
\usepackage[english]{babel}				      

\usepackage{cite}                       
\usepackage[hidelinks]{hyperref}	      
\usepackage[shortlabels]{enumitem}		  
\usepackage{colonequals}                

\usepackage{booktabs}
\usepackage{tablefootnote}
\usepackage{siunitx}
\usepackage[bottom]{footmisc}

\usepackage{bm}                         
\usepackage{mdframed}                   
\usepackage{xcolor}                     
\usepackage{arydshln}                   

\newcommand\Tstrut{\rule{0pt}{2.9ex}}         

\usepackage{amsthm}                     
\theoremstyle{definition}
\newtheorem{thm}{Theorem}

\newtheorem{prop}{Proposition}

\newtheorem{rem}{Remark}
\newtheorem*{defn}{Definition}

\newenvironment{theorem}%
  {\begin{mdframed}[backgroundcolor=gray!20!white,hidealllines=true,innerleftmargin=5pt,innerrightmargin=5pt,innertopmargin=0pt,innerbottommargin=5pt]\begin{thm}}%
  {\end{thm}\end{mdframed}}

\newcommand{\bmat}[1]{\begin{bmatrix}#1\end{bmatrix}}

\newcommand{\suchthat}{\;\ifnum\currentgrouptype=16 \middle\fi|\;}
\newcommand{\set}[2]{\left\{#1\suchthat#2\right\}}

\DeclareMathOperator{\trace}{tr}

\newcommand{\defeq}{\colonequals}
\renewcommand{\d}{\text{d}}
\newcommand{\df}{\nabla\! f}
\renewcommand{\Re}{\text{Re}}

\newcommand{\1}{\mathbf{1}}

\renewcommand{\epsilon}{\varepsilon}

\newcommand{\A}{\bm{A}}
\newcommand{\B}{\bm{B}}
\newcommand{\C}{\bm{C}}
\newcommand{\D}{\bm{D}}

\newcommand{\x}{\bm{x}}
\newcommand{\y}{\bm{y}}
\renewcommand{\u}{\bm{u}}

\newcommand{\f}{\bm{f}}
\newcommand{\n}{\bm{n}}

\newcommand{\F}{\bm{F}}

\newcommand{\G}{\bm{G}}

\newcommand{\tp}{\mathsf{T}}

\newcommand{\K}{\mathcal{K}}
\newcommand{\real}{\mathbb{R}}
\newcommand{\complex}{\mathbb{C}}

\newcommand{\rank}{\text{rank}}

\newcommand{\E}{\mathbb{E}}

\newcommand{\nhphantom}[1]{\sbox0{\ensuremath{#1}}\hspace{-\the\wd0}}

\newcommand{\makesamewidth}[3][c]{\setbox0\hbox{#2}\makebox[\the\wd0][#1]{#3}}

\makeatletter
\newcommand{\pushright}[1]{\ifmeasuring@#1\else\omit\hfill$\displaystyle#1$\fi\ignorespaces}
\newcommand{\pushleft}[1]{\ifmeasuring@#1\else\omit$\displaystyle#1$\hfill\fi\ignorespaces}
\makeatother

\setitemize{topsep=2pt,itemsep=-1pt}

\usepackage{caption}
\usepackage{thmtools,nameref}
\usepackage[capitalise,nameinlink]{cleveref}
\crefname{thm}{Theorem}{Theorems}
\crefname{lem}{Lemma}{Lemmas}
\crefname{prop}{Proposition}{Propositions}
\crefname{rem}{Remark}{Remarks}
\crefname{defn}{Definition}{Definitions}
\crefname{cor}{Corollary}{Corollaries}

\renewcommand{\epsilon}{\varepsilon}

\newcommand{\comment}[1]{\textcolor{blue}{#1}}

\title{\LARGE \bf
Absolute Stability via Lifting and Interpolation
}

\author{Bryan Van Scoy\thanks{B.~Van~Scoy is with the Department of Electrical and Computer Engineering at Miami University, Oxford, OH~45056, USA
{\tt\small bvanscoy@miamioh.edu}} \and
Laurent Lessard\thanks{L.~Lessard is with the Department of Mechanical and Industrial Engineering at Northeastern University, Boston, MA~02115, USA
{\tt\small l.lessard@northeastern.edu}\newline
L.~Lessard acknowledges support from the National Science Foundation under Grants 2136945 and 2139482.}
}

\note{}

\begin{document}

\maketitle

\begin{abstract}
We revisit the classical problem of absolute stability; assessing the robust stability of a given linear time-invariant (LTI) plant in feedback with a nonlinearity belonging to some given function class. Standard results typically take the form of sufficient conditions on the LTI plant, the least conservative of which are based on O'Shea--Zames--Falb multipliers. We present an alternative analysis based on \textit{lifting} and \textit{interpolation} that directly constructs a Lyapunov function that certifies absolute stability without resorting to frequency-domain inequalities or integral quadratic constraints. In particular, we use linear matrix inequalities to search over Lyapunov functions that are quadratic in the iterates and linear in the corresponding function values of the system in a lifted space.
We show that our approach recovers state-of-the-art results on a set of benchmark problems.
\end{abstract}

\section{Introduction}

We consider the classical absolute stability problem, also known as the Lur'e problem~\cite{lur'e}, illustrated in \cref{fig:system}. A linear time-invariant (LTI) system $G$ with internal state~$x$ is in feedback with a static nonlinearity~$\phi$.

\begin{figure}[ht]
  \centering\includegraphics{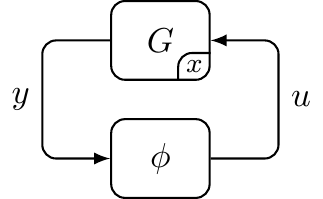}
  \caption{System $G$ with internal state $x$ in feedback with a nonlinearity $\phi$ belonging to a given function class. The \emph{absolute stability problem} is to find conditions on $G$ that ensure stability of the closed-loop system for all initial states and all nonlinearities in the function class.}
  \label{fig:system}
\end{figure}

The standard approach to absolute stability uses Zames--Falb multipliers\footnote{Also known as \textit{O'Shea--Zames--Falb multipliers}.} \cite{zames-falb} and integral quadratic constraints (IQCs) \cite{iqc}, which we review in \cref{sec:interpolation,sec:absolute-stability}. Existence of a suitable multiplier provides a guarantee of robust stability, and one can search over all finite impulse response (FIR) multipliers using convex linear matrix inequalities (LMIs) \cite{zames-falb-convex-search}. There is an abundance of literature on this topic, including tutorials such as ``Zames--Falb Multipliers: Don't Panic'' \cite{zames-falb-dont-panic} that aim to communicate the main results to a broad audience.

While searching over Zames--Falb multipliers is the state-of-the-art, it does not (at least directly) construct a Lyapunov function to certify absolute stability. Instead, the Kalman--Yakubovich--Popov (KYP) lemma \cite{kyp} is typically used to relate feasibility of an LMI to a corresponding frequency domain inequality (FDI). The main result for integral quadratic constraints \cite{iqc} is that this FDI is sufficient for absolute stability (see \cref{sec:absolute-stability}).

In contrast to the standard approach, we propose approaching absolute stability via lifting and interpolation. We lift the system to a higher-dimensional space, and then search for a Lyapunov function in this lifted space using linear matrix inequalities. The LMI is derived from a dissipation inequality whose supply rate is a combination of valid inequalities satisfied by the nonlinearity.

Our approach has several benefits. First, a feasible solution to our LMI directly provides a Lyapunov function. Since our approach is based on a simple dissipation inequality, it generalizes naturally to other settings such as robust or stochastic performance certification (see \cref{sec:beyond:stab}). Also, we relate the set of all valid multipliers for a given function class to interpolation, which provides a systematic way to construct all valid inequalities satisfied by the inputs and outputs of the nonlinearity.

\subsection{Main idea}

To certify absolute stability of the Lur'e system, we propose using a common quadratic-plus-linear Lyapunov function candidate of the form\footnote{This \emph{quadratic-plus-linear} form is similar to the  Lur'e--Postnikov Lyapunov function. The difference is that the lifted state $(\x,\f)$ is augmented to include past input and function values.}
\begin{equation}\label{eq:V}
  V(\x,\f) \defeq \x^\tp P\,\x + p^\tp \F \f
\end{equation}
where $(\x,\f)$ is the state of a \textit{lifted system}, which we describe in \cref{sec:lifting}. The matrix $P$ and vector $p$ are parameters of the Lyapunov function that we search over using a linear matrix inequality, and $\F$ is a fixed matrix.

To show that $V$ is a Lyapunov function, we will search for a quadratic-plus-linear form $\sigma_1(\y_t,\u_t,\f_t)\geq 0$ that is nonnegative over all system trajectories and that satisfies the discrete-time dissipation inequality
\begin{subequations}\label{eq:inequalities}
\begin{align}
  V(\x_{t+1},\f_{t+1}) - V(\x_t,\f_t) + \sigma_1(\y_t,\u_t,\f_t) \leq 0. \label{eq:ineq1}
\end{align}
A Lyapunov function must also be positive definite. This can be enforced by requiring $P\succ 0$ and $p>0$, but a less conservative approach is to find another nonnegative quadratic-plus-linear form $\sigma_2(\y_t,\u_t,\f_t)\geq 0$ such that
\begin{align}
\|\x_t\|^2 - V(\x_t,\f_t) + \sigma_2(\y_t,\u_t,\f_t) \leq 0.
\end{align}
\end{subequations}
Combining the two inequalities in \eqref{eq:inequalities} and using the fact that $\sigma_1$ and $\sigma_2$ are nonnegative along system trajectories,
\[
  \|\x_t\|^2 \leq V(\x_t,\f_t) \leq V(\x_{t-1},\f_{t-1}) \leq \ldots \leq V(\x_0,\f_0),
\]
so the Lyapunov function $V$ certifies absolutely stability.

We characterize the set of all nonnegative quadratic-plus-linear forms in \cref{sec:interpolation}, which we then use in \cref{sec:main} to search over all Lyapunov functions of the form \eqref{eq:V}.

\subsection{Assumptions}

To illustrate our results, we focus on the case where $G$ is a discrete-time finite-dimensional linear time-invariant single-input single-output (SISO) system, and the nonlinearity $\phi:\real\to\real$ is bounded, satisfies $\phi(0)=0$, and is \emph{monotone}, which means that
\[
  0 \leq \bigl(\phi(x) - \phi(y)\bigr) (x-y) \qquad\text{for all }x,y \in \real.
\]
Since the nonlinearity $\phi : \real\to\real$ is monotone, it is the gradient of a scalar-valued function $f : \real\to\real$, that is, $\phi = \df$. Furthermore, $f$ is a closed convex proper function since $\phi$ is monotone; see \cite{monotone}. We assume without loss of generality that $f(0) = 0$. We choose this simple setting to highlight our results, although we emphasize that our results generalize to other classes of nonlinearities.

\section{Quadratic inequalities \& interpolation}\label{sec:interpolation}

We first review standard results concerning quadratic constraints that hold between the input and output of a monotone map $\phi$. We then use interpolation conditions to construct all such inequalities that are also linear in the values of the convex function $f$ whose gradient is $\phi$.

\subsection{Background}

The conventional approach to robust stability analysis is to replace the nonlinearity $\phi$ by valid inequalities, typically quadratic, that hold between its inputs and outputs. Willems and Brockett \cite{willems-brockett} were the first to construct all valid quadratic inequalities for a monotone nonlinearity. The inequalities are described in terms of a \textit{doubly hyperdominant} matrix. A matrix $M = \{M_{ij}\}$ is doubly hyperdominant if $M_{ij}\leq 0$ for all $i\neq j$ and
\[
  \sum_i M_{ij} \geq 0 \quad\text{and}\quad \sum_j M_{ij} \geq 0 \quad\text{for all } i\text{ and }j.
\]
The following result from \cite{willems-brockett} characterizes SISO monotone functions in terms of quadratic inequalities. The result was extended to the MIMO case in \cite{zames-falb-all}.
\begin{prop}\label{prop:quad-form-finite}
  The quadratic inequality
  \begin{equation}\label{eq:quad-form}
    \bmat{u_\ell \\ \vdots \\ u_0}^\tp\! M \bmat{y_\ell \\ \vdots \\ y_0} \geq 0
  \end{equation}
  holds for all finite sets of points such that $u_t = \phi(y_t)$ for all $t=0,1,\ldots,\ell$ for some monotone function $\phi : \real\to\real$ with $\phi(0) = 0$ if and only if $M$ is doubly hyperdominant.
\end{prop}

\cref{prop:quad-form-finite} provides an exact characterization of all quadratic inequalities that hold between a finite set of input-output pairs of a monotone function. In contrast, Zames--Falb multipliers \cite{zames-falb,willems-brockett} describe a set of inequalities that hold over an \textit{infinite} horizon.

\begin{prop}\label{prop:ZF}
  Suppose $\Pi$ is a stable discrete-time LTI system. Then the quadratic inequality
  \begin{equation}\label{eq:ZF}
    \sum_{t=-\infty}^\infty u_t\,(\Pi y)_t \geq 0
  \end{equation}
  holds for all square-summable signals $u$ and $y$ such that $u_t = \phi(y_t)$ for all $t$ for some monotone function $\phi : \real\to\real$ with $\phi(0)=0$ if and only if the impulse response $\pi$ of $\Pi$ satisfies $\pi_t\leq 0$ for all $t\neq 0$ and $\sum_{t=-\infty}^\infty \pi_t \geq 0$.
\end{prop}


In contrast to \eqref{eq:quad-form} and \eqref{eq:ZF}, which only involve the inputs and outputs of the nonlinearity $\phi$, our analysis will use inequalities that also depend on the values of the convex function $f$ whose gradient is $\phi$ (which always exists since $\phi$ is monotone and one-dimensional by assumption). To this end, we characterize all tuples $(M,m)$ for which the \textit{quadratic-plus-linear} inequality
\begin{equation}\label{eq:quad-linear-form}
  \bmat{y_\ell \\ \vdots \\ y_0}^\tp\! M \bmat{u_\ell \\ \vdots \\ u_0} + m^\tp \bmat{f_\ell \\ \vdots \\ f_0} \geq 0
\end{equation}
holds for all sets of points such that $u_t = \df(y_t)$ and $f_t = f(y_t)$ for some convex function $f : \real\to\real$ with $\df(0)=0$ and $f(0)=0$. We shall see that such conditions naturally arise as the dual of conditions for convex interpolation.

\subsection{Interpolation}

We now show how to systematically derive the set of all quadratic-plus-linear inequalities \eqref{eq:quad-linear-form} using \textit{interpolation}.

\begin{defn}
  A set of points $(y_i,u_i,f_i)$ for $i=0,\ldots,\ell$ is \textit{interpolable by a convex function that passes through the origin} if there exists a convex function $f : \real\to\real$ with $f(0)=0$ and $\df(0)=0$ such that $u_i=\df(y_i)$ and $f_i = f(y_i)$ for all $i=0,\ldots,\ell$.
\end{defn}

The following result provides necessary and sufficient conditions for interpolation by a convex function that passes through the origin; see \cite[Thm.~1]{interpolation}.

\begin{prop}\label{prop:convex_interpolation}
  The set $(y_i,u_i,f_i)$ for $i=0,\ldots,\ell$ is interpolable by a convex function that passes through the origin if and only if $f_i \geq f_j + u_j\,(y_i-y_j)$, $u_i y_i\geq f_i$, and $f_i\geq 0$ for all $i,j=0,\ldots,\ell$.
\end{prop}

The interpolation conditions in \cref{prop:convex_interpolation} are \textit{linear} in the function values and \textit{quadratic} in the gradient and the point at which it is evaluated. We can therefore represent the interpolation conditions in terms of the \textit{Gram matrix} and vector of function values,
\begin{equation}\label{eq:G-f}
  \G = \bmat{y_\ell \\ \vdots \\ y_0}\,\bmat{u_\ell \\ \vdots \\ u_0}^\tp
  \quad\text{and}\quad
  \f = \bmat{f_\ell \\ \vdots \\ f_0}.
\end{equation}
In particular, a set of points is interpolable by a convex function that passes through the origin if and only the tuple $(\G,\f)$ is in the \textit{interpolation set}, defined as
\begin{multline*}
  \K = \{(\G,\f) \mid f_i \geq f_j + G_{ij} - G_{jj}, \ G_{ii}\geq f_i,\\
    \text{ and }f_i\geq 0\text{ for all }i,j \text{ and }\rank(\G)=1\}
\end{multline*}
Given a Gram matrix $\G$ and vector $\f$ in the interpolation set, we can factor $\G$ as in \eqref{eq:G-f} since it has rank one. Then, from the definition of the interpolation set, the points $(y_i,u_i,f_i)$ satisfy the interpolation conditions in \cref{prop:convex_interpolation}, so there exists a convex function that interpolates the points.
While the interpolation set exactly characterizes the Gram matrices and function values that can be generated by a convex function, the quadratic-plus-linear inequalities \eqref{eq:quad-linear-form} are naturally characterized by the \textit{dual} of the interpolation set,
\[
  \K^* = \set{(M,m)}{\trace(M^\tp\G) + m^\tp \f\geq 0, \ \forall (\G,\f)\in\K}.
\]
This is precisely the set of all tuples $(M,m)$ such that the inequality \eqref{eq:quad-linear-form} holds for all sets of points that are interpolable by a convex function that passes through the origin. Any such parameter is referred to as a \textit{multiplier}. An important implication of this is that, since the set of multipliers is a dual cone, it is always a convex set \cite{boyd}. This will allow us to efficiently search over all multipliers using convex optimization.

To find the dual cone, we take a nonnegative linear combination of all inequalities defining the interpolation set,
\[
  \sum_{i=1}^\ell \sum_{j=1}^\ell \lambda_{ij}\,(f_i-f_j-G_{ij}+G_{jj}) + \gamma_i\,(G_{ii}-f_i) + \delta_i\,f_i \geq 0
\]
where the coefficients $\lambda_{ij}$, $\gamma_i$, and $\delta_i$ are nonnegative. Assembling the coefficients into the matrix $\Lambda = \{\lambda_{ij}\}$ and vectors $\gamma=\{\gamma_i\}$ and $\delta=\{\delta_i\}$, we can write this inequality compactly as
$\trace(M^\tp \G) + m^\tp \f \geq 0$,
where the matrix $M$ and vector $m$ are given as follows:
\begin{align*}
  M &= \text{diag}(\Lambda^\tp \1) - \Lambda + \text{diag}(\gamma) \\
  m &= (\Lambda-\Lambda^\tp)\,\1 - \gamma + \delta
\end{align*}
Therefore, the dual cone of the interpolation set for the class of convex functions that pass through the origin is
\begin{multline}\label{eq:dual}
  \K^* = \{(M,m) \mid M^\tp\1 \geq 0, \ M\1 + m \geq 0, \\
  \text{and } M_{ij}\leq 0\text{ for all }i\neq j\}.
\end{multline}
We summarize this result as follows.
\begin{prop}\label{prop:quad-linear-form}
  The quadratic-plus-linear inequality \eqref{eq:quad-linear-form} holds for all sets of points $(y_i,u_i,f_i)$ for $i=0,\ldots,\ell$ that are interpolable by a convex function that passes through the origin if and only if the tuple $(M,m)$ is in the dual cone $\K^*$ in~\eqref{eq:dual}.
\end{prop}

Here, we focused on characterizing quadratic-plus-linear inequalities that hold for points that are interpolable by a scalar convex function that passes through the origin. We emphasize, however, that our construction generalizes to other function classes and other types of inequalities. We summarize this general principle linking interpolation with the set of multipliers as follows:

\begin{center}
  \parbox{0.6\linewidth}{\centering\it The set of all multipliers is the dual of the interpolation set.}
\end{center}

\begin{rem}
  Instead of using the interpolation conditions in \cref{prop:convex_interpolation} that involve function values, we could equivalently use the \textit{cyclic monotonicity} conditions that are also necessary and sufficient for convex interpolation but only involve $y_i$ and $u_i$ (but not $f_i$). Then the dual of the interpolation cone would only involve a matrix $M$ (and not a vector $m$) and would consist of all doubly hyperdominant matrices, recovering \cref{prop:quad-form-finite}.
\end{rem}

\begin{rem}
  There are several generalizations of slope-restricted nonlinearities to the multi-input multi-output case, including diagonal, repeated, or full-block cases \cite{zames-falb-all}. The MIMO generalization of monotonicity is defined by the inequality
  \[
    0 \leq \left( \phi(x)-\phi(y) \right)^\tp \left( x-y\right) \quad \text{for all }x,y\in\real^n.
  \]
  However, care must be taken because a MIMO monotone nonlinearity is \emph{not necessarily} the gradient of a convex function $f:\real^n \to \real$. If it is, then the interpolation conditions in \cref{prop:convex_interpolation} become $f_i \geq f_j + u_j^\tp (y_i-y_j)$. Otherwise, we can use the definition itself as the set of valid inequalities: $(u_i-u_j)^\tp(y_i-y_j) \geq 0$. The ability to work directly with valid inequalities means that there is no need to translate information about the function into a frequency-domain IQC. This is particularly useful in scenarios where the nonlinearity is implicitly characterized by the inequalities that it satisfies, as in iterative algorithm analysis \cite{lessard}. See \cite{interpolation} for the interpolation conditions for many other function classes.
\end{rem}

\section{Absolute stability via Zames--Falb}\label{sec:absolute-stability}

We now summarize how to use Zames--Falb multipliers to verify absolute stability of the Lur'e system.

\subsection{Frequency domain formulation}\label{sec:time-domain}

Applying Parseval's theorem, the time-domain quadratic inequality \eqref{eq:ZF} is equivalent to the frequency-domain quadratic inequality
$
  \int u(z)^* \Pi(z)\, y(z)\,\d z \geq 0
$,
where the integral is over the unit circle in the complex plane. This is an \textit{integral quadratic constraint} \cite{iqc}.

Since the signals are related by $y(z) = G(z)\,u(z)$ where $G(z)$ is the transfer function of the LTI system in \cref{fig:system}, the IQC can be written as
$\int \Pi(z)\,G(z)\,\|u(z)\|^2\,\d z \geq 0$.
The main IQC result is that the system is absolutely stable for all signals satisfying the IQC if and only if the following frequency-domain inequality\footnote{Some authors use the negative feedback convention, in which case the direction of the inequality \eqref{eq:FDI} is flipped.} holds \cite{iqc,willems-brockett}.
\begin{equation}\label{eq:FDI}
  \Re\{\Pi(z)\,G(z)\} < 0
  \quad\text{for all }z\in\complex\text{ with }|z|=1.
  \tag{FDI}
\end{equation}
Therefore, a sufficient condition for robust stability of the interconnection is the existence of a multiplier $\Pi(z)$ that satisfies the conditions in \cref{prop:ZF} and \eqref{eq:FDI}.

\subsection{Time domain formulation}

The frequency-domain inequality \eqref{eq:FDI} is not conducive to numerical verification since it must hold over an infinite number of points. By applying the Kalman--Yakubovich--Popov lemma \cite{kyp}, we can transform \eqref{eq:FDI} to an equivalent linear matrix inequality, which is more amenable to computation.

Suppose the Zames--Falb multiplier $\Pi$ has finite impulse response $\pi_t$ of length $\ell$, that is, $\pi_t = 0$ for $|t|>\ell$. Define the $2\,(\ell+1)\times 2$ transfer matrix
\[
  \Psi(z) \defeq \bmat{1 & z^{-1} & \ldots & z^{-\ell} & 0 & 0 & \ldots & 0 \\ 0 & 0 & \ldots & 0 & 1 & z^{-1} & \ldots & z^{-\ell}}^\tp
\]
and define the matrix
\begin{equation}\label{eq:M-ZF}
  N = \bmat{
    \pi_0 & \pi_1 & \ldots & \pi_\ell \\
    \pi_{-1} & 0 & \ldots & 0 \\
    \vdots & \vdots & \ddots & \vdots \\
    \pi_{-\ell} & 0 & \ldots & 0}.
\end{equation}
Then, \eqref{eq:FDI} can be written as
\[
  \bmat{G(z) \\ 1}^* \Psi(z)^* \bmat{0 & N^\tp \\ N & 0} \Psi(z) \bmat{G(z) \\ 1} < 0.
\]
Let $(\A,\B,\C,\D)$ be a minimal realization of the system
\begin{equation}\label{eq:lifted-system}
  \left[\begin{array}{c|c} \A & \B \\ \hline\Tstrut \C & \D \end{array}\right]
  =
  \Psi(z) \bmat{G(z) \\ 1}.
\end{equation}

Applying the KYP lemma (specifically, the positive-real lemma), we see that \eqref{eq:FDI} is equivalent to the existence of $P = P^\tp\succ 0$ satisfying
\begin{equation}\label{eq:LMI}
\addtolength{\arraycolsep}{-3pt}
  \bmat{\A^\tp P \A - P & \A^\tp P \B \\ \B^\tp P \A & \B^\tp P \B} + \frac{1}{2} \bmat{\C & \D}^\tp\! \bmat{0 & N^\tp \\ N & 0}\! \bmat{\C & \D} \prec 0. \tag{LMI}
\end{equation}
This is a linear matrix inequality, which can be readily solved using interior-point solvers, for example.

To summarize, if \eqref{eq:LMI} is feasible, then we have found an FIR multiplier of length $\ell$ that certifies \eqref{eq:FDI}, which then implies absolute stability.

We can also verify absolute stability directly from the feasibility of \eqref{eq:LMI} as follows. Suppose that \eqref{eq:LMI} is feasible for some positive definite matrix $P$ and some FIR multiplier $\Pi$ of length $\ell$. Multiply the LMI on the right and left by $(\x_t, u_t)$ and its transpose, respectively, where $\x_t$ is the state of \eqref{eq:lifted-system} at time $t$. Doing so, we obtain the dissipation inequality
\begin{equation}\label{eq:ZF-dissipation}
  \x_{t+1}^\tp P\, \x_{t+1} - \x_t^\tp P\, \x_t + \bmat{u_t \\ \vdots \\ u_{t-\ell}}^\tp\! N \bmat{y_t \\ \vdots \\ y_{t-\ell}} < 0.
\end{equation}
We now sum the dissipation inequality over $t$ from $0$ to~$T$. The first two terms form a telescoping sum so that
\begin{multline*}
  0 > \x_{T+1}^\tp P\,\x_{T+1} - \x_0^\tp P\,\x_0 \\
  + \bmat{u_T \\ \vdots \\ u_1 \\ u_0}^\tp \bmat{\pi_0 & \ldots & \pi_{T-1} & \pi_T \\ \vdots & \ddots & \vdots & \vdots \\ \pi_{1-T} & \ldots & \pi_0 & \pi_1 \\ \pi_{-T} & \ldots & \pi_{-1} & \pi_0} \bmat{y_T \\ \vdots \\ y_1 \\ y_0}.
\end{multline*}
The conditions on the multiplier ensure that the Toeplitz matrix is doubly hyperdominant, so the quadratic form is nonnegative from \cref{prop:quad-form-finite}. We conclude that $\x_T^\tp P\,\x_T < \x_0^\tp P\,\x_0$ for all $T>0$, so the Lur'e system is absolutely stable since $P$ is positive definite.

To motivate our approach, suppose that we required $N$ to be doubly hyperdominant instead of having the structure in \eqref{eq:M-ZF}. Then we could apply \cref{prop:quad-form-finite} directly to the dissipation inequality \eqref{eq:ZF-dissipation} to obtain $\x_{t+1}^\tp P\,\x_{t+1} < \x_t^\tp P\,\x_t$ for all $t$. Then $\x^\tp P\,\x$ is a Lyapunov function for the system that verifies absolute stability. While this approach directly constructs a Lyapunov function, we find that it is quite conservative for small values of $\ell$. If we also include the function values in the analysis (as we describe in the following section), however, we directly construct a Lyapunov function using values for $\ell$ similar to those needed using Zames--Falb multipliers.

\section{Main result}\label{sec:main}

We now describe our main result, which is an LMI that directly searches over Lyapunov functions that are quadratic in the iterates of the system and linear in the corresponding values of the convex function $f$ whose gradient is $\phi$. Existence of such a Lyapunov function certifies absolute stability of the system. As we will show in \cref{sec:numerical-validation}, our analysis matches state-of-the-art results while also explicitly constructing a Lyapunov function.

\subsection{Lifting}\label{sec:lifting}

The main idea behind our analysis is to augment the state of the system in \cref{fig:system} so that it becomes more amenable to analysis. This procedure is known as a \textit{lifting}, and the transformed system is called the \textit{lifted system}. The idea of lifting is commonly used in polynomial optimization~\cite{sos-review}, kernel methods~\cite{kernal-methods}, and Koopman theory~\cite{brunton}, where an intractable problem is lifted to a higher-dimensional space in which it becomes tractable.

\cref{fig:lift} provides an overview of the lifting procedure, where a lifting maps trajectories of the original system to those of the lifted system. While direct analysis of the original system is conservative, the analysis becomes less conservative in the lifted space; in particular, feasibility of a linear matrix inequality for the lifted system provides a Lyapunov function that guarantees absolute stability of the original system.


\begin{figure}[htb]
  \centering\includegraphics{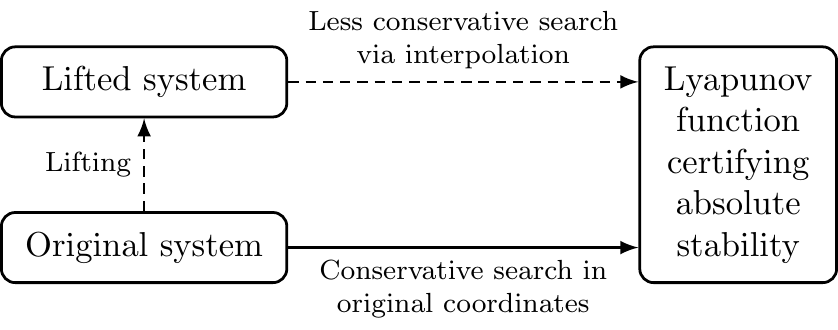}
  \caption{Overview of the lifting procedure. The state of the original system in \cref{fig:system} is lifted to a higher-dimensional space in which absolute stability can be verified using a linear matrix inequality.}
  \label{fig:lift}
\end{figure}

Given a nonnegative integer $\ell$, referred to as the \textit{lifting dimension}, the lifted system is a dynamical system that describes the evolution of the the \textit{lifted iterates}
\begin{align}\label{eq:lifted-iterates}
  \y_t &= \bmat{y_t \\ \vdots \\ y_{t-\ell}}\!, &
  \u_t &= \bmat{u_t \\ \vdots \\ u_{t-\ell}}\!, &
  \f_t &= \bmat{f_t \\ \vdots \\ f_{t-\ell}}\! &
\end{align}
where $u_t = \df(y_t)$ are the gradients and $f_t = f(y_t)$ the corresponding function values of the original system. These trajectories satisfy the \textit{lifted dynamics}
\begin{equation}\label{eq:lifted-sys}
  \bmat{\x_{t+1} \\ \y_t \\ \u_t} = \bmat{\A & \B \\ \C & \D} \bmat{\x_t \\ u_t} \quad\text{and}\quad
  \F \f_{t+1} = \F_+ \f_t
\end{equation}
where the state-space matrices $(\A,\B,\C,\D)$ are defined in \eqref{eq:lifted-system} and the matrices describing the evolution of the function values are defined as
\begin{equation*}
  \F = \bmat{0_{\ell\times 1} & I_\ell} \qquad\text{and}\qquad
  \F_+ = \bmat{I_\ell & 0_{\ell\times 1}}.
\end{equation*}
The key feature of the lifted system is that each element of the lifted iterates satisfies the uncertainty, thereby capturing information about the uncertainty across multiple time steps of the original system. Larger lifting dimensions produce higher-dimensional lifted iterates that are able to capture more information about the uncertainty and therefore produce a less conservative analysis.

\subsection{Linear matrix inequality}

We now construct an LMI that uses the set of multipliers in \eqref{eq:dual} to search over the class of quadratic-plus-linear Lyapunov functions in \eqref{eq:V} for the lifted system \eqref{eq:lifted-sys} that satisfy the dissipation and positivity inequalities \eqref{eq:inequalities}.

\begin{theorem}\label{thm}
  Consider the system in \cref{fig:system} with state dimension $n$. Given a lifting dimension $\ell$, construct state-space matrices $(\A,\B,\C,\D,\F,\F_+)$ for the lifted system \eqref{eq:lifted-system}, and denote the lifted state dimension as  $\n$.
  
  If there exists a symmetric $P\in\real^{\n\times\n}$, vector $p\in\real^{\ell+1}$, and multipliers $(M_1,m_1)$ and $(M_2,m_2)$ in the dual cone $\K^*$ defined in \eqref{eq:dual} such that the following linear matrix inequality is feasible:
  \begin{subequations}\label{lmi}
    \begin{align*}
      \addtolength{\arraycolsep}{-4pt}
      \bmat{\A^\tp P \A\!-\! P & \A^\tp P \B \\ \B^\tp P \A & \B^\tp P \B}\! + \frac{1}{2} \bmat{\C & \D}^\tp\! \bmat{0 & M_1^\tp \\ M_1 & 0}\! \bmat{\C & \D} &\preceq 0 \\
      (\F_+ - \F)^\tp p + m_1 &\leq 0 \\
      \addtolength{\arraycolsep}{-3pt}
      \bmat{I-P & 0 \\ 0 & 0} + \frac{1}{2} \bmat{\C & \D}^\tp\! \bmat{0 & M_2^\tp \\ M_2 & 0} \!\bmat{\C & \D} &\preceq 0 \\
      -\F^\tp p + m_2 &\leq 0
    \end{align*}
  \end{subequations}
  then for any initial state $x_0\in\real^n$ and any bounded monotone function $\phi : \real\to\real$ with $\phi(0)=0$, the system in \cref{fig:system} is Lyapunov stable.
\end{theorem}

\begin{proof}
  Consider a trajectory of the original system along with its corresponding lifted iterates \eqref{eq:lifted-iterates}. Multiply the first LMI on the right and left by the vector $(\x_t,u_t)$ and its transpose, respectively, and add the result to the inner product of the second inequality with $\f_t$ to obtain
  \begin{equation*}
    V(\x_{t+1},\f_{t+1}) - V(\x_t,\f_t) \leq 0,
  \end{equation*}
  where the function $V$ is defined in \eqref{eq:V} and we used that the inequality $\y_t^\tp M_1 \u_t + m_1^\tp \f_t \geq 0$ holds from \cref{prop:quad-linear-form} since the lifted iterates are interpolable by a convex function that passes through the origin and the multiplier $(M_1,m_1)$ is in the dual cone $\K^*$. A similar computation with the last two inequalities gives
  \begin{equation*}
    \|\x_t\|^2 - V(\x_t,\f_t) \leq 0.
  \end{equation*}
  Combining these, we have that
  \[
    \|\x_t\|^2 \leq V(\x_t,\f_t) \leq V(\x_{t-1},\f_{t-1}) \leq \ldots \leq V(\x_0,\f_0),
  \]
  which proves that the system is Lyapunov stable.
\end{proof}





\section{Numerical validation}\label{sec:numerical-validation}

We tested our approach on the examples from \cite[Table~I]{zames-falb-convex-search}, which are shown in \cref{tab:ex_plants} below.

\begin{table}[ht]
    \centering
    \renewcommand{\arraystretch}{1.2}
    \begin{tabular}{cc} \toprule
        Ex. &  Plant $G(z)$ \\ \midrule
        1 & $\frac{0.1z}{z^2-1.8z+0.81}$\\
        2 & $\frac{z^3-1.95z^2+0.9z+0.05}{z^4-2.8z^3+3.5z^2-2.412z+0.7209}$ \\
        3 & $-\frac{z^3-1.95z^2+0.9z+0.05}{z^4-2.8z^3+3.5z^2-2.412z+0.7209}$ \\
        4 & $\frac{z^4-1.5z^3+0.5z^2-0.5z+0.5}{4.4z^5-8.957z^4+9.893z^3-5.671z^2+2.207z-0.5}$\\
        5 & $\frac{-0.5z+0.1}{z^3-0.9z^2+0.79z+0.089}$\\
        6 & $\frac{2z+0.92}{z^2-0.5z}$\\
        7 & $\frac{1.341z^4-1.221z^3+0.6285z^2-0.5618z+0.1993}{z^5-0.935z^4+0.7697z^3-1.118z^2+0.6917z-0.1352}$\\ \bottomrule
    \end{tabular}
    \caption{Numerical examples from \cite[Table I]{zames-falb-convex-search}.}
    \label{tab:ex_plants}
\end{table}

In each example, the system $G(z)$ is in negative feedback with a slope-restricted nonlinearity in $(0,\alpha)$, so the nonlinearity satisfies $0 \leq \frac{\phi(x)-\phi(y)}{x-y} \leq \alpha$ for all $x,y\in\real$. The goal is to find the largest $\alpha$ such that we can ensure absolute stability of the closed-loop system. To apply \cref{thm}, we first use a loop-shifting transformation. Namely, the following statements are equivalent:
\begin{enumerate}[i)]
    \item Absolute stability of $G$ in negative feedback with a slope-restricted nonlinearity in $(0,\alpha)$.
    \item Absolute stability of $-(1+\alpha G)$ in positive feedback with a monotone nonlinearity.
\end{enumerate}
We find the largest feasible $\alpha$ by performing a bisection search over $\alpha$ and repeatedly applying \cref{thm} to test absolute stability. For each example, we were able to recover state-of-the-art results found using the IQC approach with FIR Zames--Falb multiplier search as in \cite{zames-falb-convex-search}, and the lifting dimension required was comparable to the number of Zames--Falb multipliers needed in the IQC approach. 
Results are shown in \cref{tab:ex_results}.
\begin{table}[ht]
    \centering
    \begin{tabular}{cS[table-format=3.4]cc} \toprule
        Ex. & {$\alpha$}  & {$(n_b,n_f)$} & {$\ell$} \\ \midrule
        1 & 12.9960 & $(1,0)$ & 1 \\
        2 &  0.8027 & $(1,4)$ & 4 \\
        3 &  0.3054 & $(0,1)$ & 1 \\
        4 &  3.8240 & $(0,4)$ & 4 \\
        5 &  2.4475 & $(0,1)$ & 1 \\
        6 &  0.9114 & $(1,2)$ & 2 \\
        7 &  0.4347\tablefootnote{It was noted in \cite{zames-falb-convex-search} that $\alpha =0.4922$ is achieved if $n_b=n_f=25$. Our numerical solver was error-prone and unreliable for such large LMIs, so we simulated a suboptimal value instead.}  & $(3,3)$ & 3 \\ \bottomrule
    \end{tabular}
    \caption{Comparison of the IQC approach vs.\ our lifting approach on the examples from \cref{tab:ex_plants}. Both achieve the same $\alpha$. Here, $n_b$ and $n_f$ are the fewest causal and anticausal Zames--Falb multipliers needed, and $\ell$ is the smallest lifting dimension needed to achieve $\alpha$.}
    \label{tab:ex_results}
\end{table}

We now explain one of the examples in greater detail. Consider Example~6 with $\alpha=0.91$. We begin with any realization of $-(1+\alpha G)$. For example, we could pick
\[
-(1+\alpha G) = \left[\begin{array}{c|c}
  A & B \\ \hline\Tstrut C & D
\end{array}\right]
=
\left[\begin{array}{cc|c}
  0.5 & 0 & 2 \\
  1.0 & 0 & 0 \\ \hline\Tstrut
  -0.91 & -0.4186 & -1
\end{array}\right]\!.
\]
Construct the lifted realization in \eqref{eq:lifted-sys} with lifted state $\x_t = (x_{t-2},u_{t-2},u_{t-1})$ and lifting dimension $\ell=2$. Solving the LMI from \cref{thm}, we obtain the feasible point
\begin{small}
\begin{gather*}
P = \bmat{
  1.4483 &  -0.2173 &  -2.4073 &  -2.4262\\
 -0.2173 &   0.8523 &  -2.6369 &   0.1214\\
 -2.4073 &  -2.6369 &   2.4142 &  -1.5938\\
 -2.4262 &   0.1214 &  -1.5938 &   0.4756},
 \\
M_1 = \bmat{
  8.6813 &  -8.6813 &  -0.0000\\
 -0.0000 &   5.8115 &  -5.8115\\
 -2.5025 &  -0.0000 &   2.5278},
\\
 M_2 = \bmat{
 11.2412 &  -3.3521 &  -1.6564\\
 -1.6595 &  10.6892 &  -2.7451\\
 -1.3351 &  -1.5047 &   5.9290},
\\
p = \bmat{
 -6.1534\\
 -3.2837}\!,
 \;
 m_1 = \bmat{
 -6.1788\\
  2.8698\\
  3.2837}\!,
  \;
  m_2 = \bmat{
 -8.2467\\
 -5.8325\\
 -0.0000}\!.
\end{gather*}
\end{small}

Consequently, a Lyapunov function that certifies robust stability of Example~6 with $\alpha=0.91$ is given by
\[
  V(\x_t,\f_t) = \bmat{x_{t-2}\\u_{t-2}\\u_{t-1}}^\tp P \bmat{x_{t-2}\\u_{t-2}\\u_{t-1}}
+ p^\tp \bmat{f_{t-2} \\ f_{t-1}}.
\]
The function $V$ satisfies \eqref{eq:inequalities} for all trajectories of the system and all slope-restricted  $\phi$ in the sector $(0,\alpha)$. Note that $P$ is an indefinite matrix, yet the feasibility of the LMI ensures that $V$ is positive over all trajectories. Note also that $V$ is realization-dependent, so using a different realization $(A,B,C,D)$ would transform $P$ accordingly.

\begin{rem}
Different (sometimes simpler) Lyapunov functions can be sought by imposing sparsity constraints on the LMI variables. For example, the LMI from \cref{thm} remains feasible if we set $p=0$ (no function value dependence), but increase the lifting dimension to $\ell=3$.
\end{rem}

\section{Beyond absolute stability}\label{sec:beyond:stab}

Our proposed approach using lifting and interpolation has the benefit of producing a stability certificate in the form of a Lyapunov function and associated dissipation inequality. Moreover, our approach operates entirely in the time domain using finite sums, which leads to simple and intuitive proofs. We only treated the case of absolute stability for discrete-time SISO systems, but our approach generalizes easily to other use cases. We now discuss a few of these potential extensions.

\paragraph{Continuous time.}
Lifting works analogously in continuous time if we replace the time lags with integrals. For example, where the lifted state might include a lagged input $u_{t-k}$ in discrete time, we would  use the input $\idotsint u(t)$ (integrated $k$ times) in continuous time. We still use a Lyapunov function of the form \eqref{eq:V}, except the associated dissipation inequality \eqref{eq:ineq1}  becomes
\[
0 \geq \frac{\mathrm{d}}{\mathrm{d}t}V(\x(t)) + \sigma(\y(t),\u(t),\f(t)),
\]
which leads to $\left[\begin{smallmatrix} \A^\tp P \A - P & \A^\tp P \B \\ \B^\tp P \A & \B^\tp P \B \end{smallmatrix}\right] \mapsto \left[\begin{smallmatrix} \A^\tp P + P \A & P \B \\ \B^\tp P & 0 \end{smallmatrix}\right]$ in \cref{thm}.

\paragraph{Exponential stability.}
\cref{thm} can be generalized to certify exponential stability by introducing a scaling factor $\rho$ in \eqref{eq:inequalities}. We state the result below.

\begin{prop}[exponential stability]\label{prop:ext1}
  If the following dissipation inequalities are satisfied:
  \begin{align*}
  V(\x_{t+1},\f_{t+1}) - \rho^2 V(\x_t,\f_t) + \sigma_1(\y_t,\u_t,\f_t) & \leq 0 \\
  \|\x_t\|^2 - V(\x_t,\f_t) + \sigma_2(\y_t,\u_t,\f_t) & \leq 0
  \end{align*}
  then $\|\x_t\| \leq V(\x_0,\f_0)^{1/2} \rho^t$ for all $t\geq 0$. In particular, if $\rho < 1$, then $G$ is robustly exponentially stable.
\end{prop}

If the task is to find the fastest certifiable exponential rate (smallest possible $\rho$), this can be efficiently achieved by performing a bisection search on $\rho$ \cite{lessard}.

\paragraph{Robust performance.}
The lifting approach can be used to certify a variety of performance criteria beyond absolute stability. Essentially, any performance measure compatible with dissipativity can be used. For example, consider the system of \cref{fig:system_aug}, where the LTI system $G$ is augmented to include a performance channel $w\to z$. Similar to how the dissipation inequalities \eqref{eq:inequalities} lead to proving absolute stability via \cref{thm}, we can state analogous results for quadratic performance and performance under stochastic disturbances.

\begin{figure}[ht]
  \centering\includegraphics{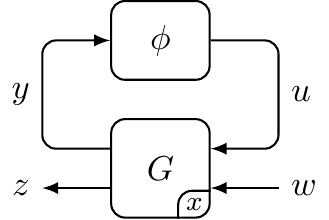}
  \caption{System $G$ with internal state $x$ in feedback with a nonlinearity $\phi$ and a performance channel $w\to z$.}
  \label{fig:system_aug}
\end{figure}

\begin{prop}[quadratic performance]\label{prop:ext2}
If the following dissipation inequalities are satisfied for all signals $w$:
\begin{align*}
V(\x_{t+1},\f_{t+1}) - V(\x_t,\f_t) + \sigma_1(\y_t,\u_t,\f_t)   &\leq \sigma_p(w_t,z_t)\\
\|\x_t\|^2 - V(\x_t,\f_t) + \sigma_2(\y_t,\u_t,\f_t)  & \leq 0
\end{align*}
then $\sum_{t=0}^T \sigma_p(w_t,z_t) \geq 0$ for all $T\geq 0$.
For example, if $\sigma_p(w_t,z_t) \defeq \gamma^2 \|w_t\|^2-\|z_t\|^2$, then $G$ has a robust $\ell_2$ gain from $w$ to $z$ of $\gamma$.
\end{prop}

\begin{prop}[stochastic performance]\label{prop:ext3}
Suppose $w_t$ is i.i.d. zero-mean random noise with $\mathrm{cov}(w_t) = \Sigma$, and 
the following dissipation inequalities are satisfied:
\begin{align*}
V(\x_{t+1},\f_{t+1}) - V(\x_t,\f_t) + \sigma_1(\y_t,\u_t,\f_t)  + \|z_t\|^2 &\leq 0\\
- V(\x_t,\f_t) + \sigma_2(\y_t,\u_t,\f_t)  & \leq 0 \\
\trace\left( P \B_w \Sigma \B_w^\tp \right) & \leq \gamma^2,
\end{align*}
where $\B_w$ comes from the lifted state equations for $G$, analogous to \eqref{eq:lifted-sys}. Namely, $\x_{t+1} = \A \x_t + \B u_t + \B_w w_t$.
Then, $\limsup_{T\to\infty}\E \frac{1}{T} \sum_{t=0}^{T-1}\|z_t\|^2 \leq \gamma^2$. This allows us to certify a robust $\mathcal{H}_2$ performance. 
\end{prop}

\crefrange{prop:ext1}{prop:ext3} have similar analogs with IQCs \cite{scherer_iqc_2016,scherer}. A recent application of similar results to the analysis of iterative optimization algorithms can be found in \cite{BVS_LL_speed_robustness_tradeoff}.

\appendix

\bibliographystyle{abbrv}
{\footnotesize\bibliography{references}}

\end{document}